\documentclass[11pt,reqno,a4paper]{amsart}

\usepackage[latin1]{inputenc}
\usepackage{tikz}
\usetikzlibrary{shapes,arrows}
\usetikzlibrary{arrows.meta}
\usepackage{forest}
\usepackage{tikz-qtree}

\usetikzlibrary{arrows,shapes,positioning,shadows,trees}

\usepackage[hidelinks]{hyperref}

\usepackage{etoolbox}
\usepackage{amsmath}
\usepackage{enumerate}
\usepackage{amssymb}
\usepackage{amscd}
\usepackage{amsthm}
\usepackage{amsfonts}
\usepackage{graphicx}
\usepackage[all,cmtip]{xy}
\usepackage{enumitem}
\usepackage{xcolor}

\patchcmd{\subsection}{-.5em}{.5em}{}{}
\patchcmd{\subsubsection}{-.5em}{.5em}{}{}

\usepackage{enumitem}

\makeatother
\usepackage{hyperref}

\numberwithin{equation}{section}

\newcommand{\cR}{\mathcal{R}}

\newcommand{\bR}{\mathbb{R}}

\newcommand{\bT}{\mathbb{T}}

\newcommand{\bZ}{\mathbb{Z}}

\newcommand{\ra}{\rightarrow}

\newcommand{\qand}{\quad \textrm{and} \quad}

\def\acts{\curvearrowright}
\newcommand\subsetsim{\mathrel{%
\ooalign{\raise0.2ex\hbox{$\subset$}\cr\hidewidth\raise-0.8ex\hbox{\scalebox{0.9}{$\sim$}}\hidewidth\cr}}}
\newcommand{\eps}{\varepsilon}

\theoremstyle{theorem}
\newtheorem{theorem}{Theorem}[section]
\newtheorem{corollary}[theorem]{Corollary}
\newtheorem{proposition}[theorem]{Proposition}
\newtheorem{lemma}[theorem]{Lemma}
\newtheorem{scholium}[theorem]{Scholium}

\theoremstyle{definition}
\newtheorem{definition}[theorem]{Definition}
\newtheorem{remark}[theorem]{Remark}

\newtheorem{example}{Example}[section]

\tikzstyle{decision} = [diamond, draw, fill=blue!20, 
    text width=4.5em, text badly centered, node distance=3cm, inner sep=0pt]
\tikzstyle{block} = [rectangle, draw, fill=blue!20, 
    text width=5em, text centered, rounded corners, minimum height=2em]
\tikzstyle{line} = [draw, -latex']
\tikzstyle{cloud} = [draw, ellipse,fill=red!20, node distance=3cm,
    minimum height=2em]

\renewcommand\labelenumi{(\roman{enumi})}
\renewcommand\theenumi\labelenumi

\usepackage{changepage}
\usepackage{mathtools}
\usepackage{graphicx}
\usepackage{calc}

\DeclarePairedDelimiterX\Set[2]{\{}{\}}{#1\,\delimsize\vert\,#2}

\begin{document}
\bibliographystyle{plain} 

\title{Sets of transfer times with small densities}

\author{Michael Bj\"orklund}
\address[Michael Bj\"orklund]{Department of Mathematics, Chalmers, Gothenburg, Sweden}
\email{micbjo@chalmers.se}

\author{Alexander Fish}
\address[Alexander Fish]{School of Mathematics and Statistics F07, University of Sydney, NSW 2006, Australia}
\email{alexander.fish@sydney.edu.au}

\author{Ilya D. Shkredov}
\address[Ilya D. Shkredov]{Steklov Mathematical Institute, ul. Gubkina,
8, Moscow, Russia, 119991, and
	IITP RAS, Bolshoy Karetny per. 19, Moscow, Russia, 127994, and
	MIPT, Institutskii per. 9, Dolgoprudnii, Russia, 141701}
\email{ilya.shkredov@gmail.com}

\date{}

\begin{abstract}
We consider in this paper the set of transfer times between two measurable subsets of positive measures in an ergodic probability measure-preserving system of 
a countable abelian group. If the lower asymptotic density of the transfer times is small, then we prove this set must be either periodic 
or Sturmian. Our results can be viewed as ergodic-theoretical extensions of some classical sumset theorems  in compact abelian groups due to Kneser. 
Our proofs are based on a correspondence principle for action sets which was developed previously by the first two authors.
\end{abstract}

\keywords{Return times, inverse theorems, sumsets}

\subjclass[2010]{Primary:  37A45; Secondary: 28D05, 11B13  }
\maketitle

\section{Introduction}

Throughout this paper, we shall assume that
\vspace{0.2cm}
\begin{itemize}
\item $G$ is a countable and discrete abelian group.
\vspace{0.2cm}
\item $(X,\mu)$ is a standard probability measure space, endowed with an \emph{ergodic} probability measure-preserving action of $G$.
\vspace{0.2cm}
\item $(F_n)$ is a sequence of finite subsets of $G$ with the property that for every bounded measurable function $\varphi$ on $X$, there 
exists a $\mu$-conull subset $X_\varphi \subset X$ such that
\begin{equation}
\label{pointwise}
\lim_n \frac{1}{|F_n|} \sum_{g \in F_n} \varphi(gx) = \int_X \varphi \, d\mu, \quad \textrm{for all $x \in X_\varphi$}.
\end{equation}
\end{itemize}
\vspace{0.2cm}

If $C$ is a subset of $G$, its \emph{lower asymptotic density $\underline{d}(C)$ with respect to $(F_n)$} is given by
\begin{equation}
\label{def_lad}
\underline{d}(C) = \varliminf_n \frac{|C \cap F_n|}{|F_n|}.
\end{equation}
The \emph{set of transfer times $\cR_{A,B}$} between two measurable subsets $A$ and $B$ of $X$ is defined by
\begin{equation}
\label{def_RAB}
\cR_{A,B} = \big\{ g \in G \, \mid \, \mu(A \cap g^{-1}B) > 0 \big\}.
\end{equation}
We set $\cR_A = \cR_{A,A}$, and refer to $\cR_A$ as the \emph{set of return times to the set $A$}. \\

To briefly give a flair of the type of results that we are after in this paper, let us first consider the case when $G$ is a \emph{finite} abelian group,
$X = G$ (where $G$ acts on $X$ by translations, preserving the normalized counting measure $\mu$ on $X$) and $F_n = G$ for all $n$. 
If $A$ and $B$ are non-empty subsets of $X$, then we note that $\cR_{A,B} = BA^{-1}$, the difference set of $A$ and $B$ in $G$. A fundamental 
line of research in additive combinatorics is concerned with the structure of "small" difference sets in the group $G$; in particular, one wishes to 
understand to which extent the smallness forces the difference set $AB^{-1}$ to be a coset (or a union of "few" cosets) of a subgroup of $G$. 
For instance, it follows from the work of Kneser \cite{Kneser} that if 
\begin{equation}
\label{subgroup}
A = B \qand |AA^{-1}| < \frac{3}{2}|A|,
\end{equation}
then $AA^{-1}$ must be a subgroup of $G$, and if 
\[
|BA^{-1}| < |A| + |B|,
\]
then $BA^{-1}$ is invariant under a subgroup $G_o$ of $G$ such that
\begin{equation}
\label{passtoquotient}
|BA^{-1}| = |AG_o| + |BG_o| - |G_o|.
\end{equation}
In particular, if we denote by $C$ and $D$ the images of $A$ and $B$ under the canonical quotient map $\pi : G \ra H := G/G_o$, then 
\[
A \subset \pi^{-1}(C) \qand B \subset \pi^{-1}(D) 
\]
and
\begin{equation}
\label{id}
\frac{|BA^{-1}|}{|G|} = \frac{|(BG_o)(AG_o)^{-1}|}{|G_o||H|} = \frac{|DC^{-1}|}{|H|} = \frac{|C|}{|H|} + \frac{|D|}{|H|} - \frac{1}{|H|}.
\end{equation}
It turns out that pairs $(C,D)$ of subsets in $H$ which satisfy \eqref{id} are quite structured (we refer the reader to Chapter 3.3.2 in \cite{Ham} for a survey about results in this direction), whence the pair $(A,B)$ in $G$ is "controlled" by a "structured" pair $(C,D)$ in the quotient group $H$ in a very precise way. Our aim in this paper is to show that these phenomena extend to the ergodic-theoretical setting described above, where $\underline{d}$ plays the role of the counting measure. \\

Before we proceed to our main results, let us briefly discuss the ergodic-theoretical analogue of "control" discussed above. If $(Y,\nu)$ is another standard probability measure space, endowed with an ergodic probability measure-preserving action of $G$, then we say that $(Y,\nu)$ is a \emph{factor} of $(X,\mu)$ if there exists a measurable $G$-invariant $\mu$-conull subset $X' \subseteq X$
and a $G$-equivariant measurable map $\pi : X' \ra Y$ such that $\pi_*(\mu \mid_{X'}) = \nu$. If we wish to suppress the dependence
on the set $X'$, we shall simply write $\pi : (X,\mu) \ra (Y,\nu)$ for this map. We note that if $C$ and $D$ are measurable subsets of $Y$ 
such that 
\[
A \subset \pi^{-1}(C) \qand B \subset \pi^{-1}(D), \quad \textrm{modulo $\mu$-null sets},
\]
then $\cR_{A,B} \subseteq \cR_{C,D}$. We further say that the pair $(C,D)$ \emph{controls $(A,B)$} if $\cR_{A,B} = \cR_{C,D}$. If we wish
to emphasize the dependence on $\pi$, we say that $(C,D)$ $\pi$-controls $(A,B)$.  \\

Let us now explain the framework of this paper. Roughly speaking, we are motivated by the following vague questions: If $A$ and $B$ are measurable subsets of $X$ with positive $\mu$-measures, then
\vspace{0.2cm}
\begin{adjustwidth*}{0.5in}{0.5in}
\vspace{0.05cm}
\begin{itemize}
\item how small can the lower asymptotic density of $\cR_{A,B}$ be?
\item if $\cR_{A,B}$ is small, what can we say about structure of $\cR_{A,B}$?
\item if $\cR_{A,B}$ is small, must the pair $(A,B)$ be controlled by another pair in a "small" factor of $(X,\mu)$?
\end{itemize}  
\end{adjustwidth*}
\vspace{0.2cm}
Below we shall answer these questions for different notions of smallness (with respect to $\underline{d}$), and provide examples which show that our results are optimal
in the settings under study. 

\begin{remark}[Standing assumptions]
Let us first make a few basic observations. We note that if 
$\mu(A) + \mu(B) > 1$, then $\mu(A \cap g^{-1}B) = 1$ for all $g \in G$, whence $\cR_{A,B} = G$. If $\mu(A) + \mu(B) = 1$,
then either $\cR_{A,B} = G$ or there exists $g_o \in G$ such that $\mu(A \cap g_o^{-1}B) = 0$. In the latter case, $B = g_o^{-1}A^c$ 
modulo $\mu$-null sets, so if denote by $H$ the $\mu$-essential stabilizer of $A$, then, for every $g \in G$,
\[
\mu(A \cap g^{-1}B) = \mu(A \cap (g_o g)^{-1}A^c) = 0 \iff g_o g \in H,
\]
whence $\cR_{A,B} = G \setminus g_o^{-1} H$. We conclude that if $\mu(A) + \mu(B) = 1$, then either $\cR_{A,B}$ is $G$ or equal to the complement of
a single coset of some subgroup of $G$. Hence, to get non-trivial results, we shall henceforth always assume that
\begin{equation}
\label{standass}
\mu(A) + \mu(B) < 1.
\end{equation}
In particular, if $A = B$, we shall assume that $\mu(A) < 1/2$. 
\end{remark}

\begin{remark}[Concerning novelty]
All of the results in this paper are new already in the case when $G = (\bZ,+)$ and $F_n = \{1,\ldots,n\}$ (this sequence satisfies \eqref{pointwise} by Birkhoff's Ergodic Theorem). However, we stress that 
we in general do \emph{not} need to assume that the sequence $(F_n)$ is F\o lner (asymptotically invariant) in the group $G$. For instance, in the case of $(\bZ,+)$, our results below will also
apply to the rather sparse sequence 
\[
F_n = \{ k \sqrt{2} + k^{5/2} \, \mid \, k = 1,\ldots,n  \}, \quad \textrm{for $n \geq 1$},
\]
which is far from being a F\o lner sequence. We refer the reader to \cite{BW} for more examples of this type. 
\end{remark}

\subsection{Main results}
The main point of our first theorem is that if the lower asymptotic density of $\cR_A$ is small enough, then the set of transfer times $\cR_A$ is in fact a 
subgroup of $G$ (this is the ergodic-theoretical analogue of \eqref{subgroup} described above for \emph{finite} $G$). 
\begin{theorem}
\label{thm1}
For all measurable subsets $A$ and $B$ of $X$ with positive $\mu$-measures, we have
\[
\underline{d}(\cR_{A,B}) \geq \max(\mu(A),\mu(B)).
\]
Furthermore, suppose that $\underline{d}(\cR_A) < \frac{3}{2} \mu(A)$. Then there exists a finite-index subgroup $G_o < G$ with 
index $[G : G_o] \leq \frac{1}{\mu(A)}$, such that $\cR_A = G_o$.
\end{theorem}

The proof of this theorem provides additional information about the set $A \subset X$ which we do not state here 
(see Theorem \ref{thm2} below for a generalization). Instead, we discuss the sharpness of the assumptions in the
theorem, namely the constant $\frac{3}{2}\mu(A)$ and the ergodicity of $G \acts (X,\mu)$.

\begin{example}[The constant $\frac{3}{2} \mu(A)$ is optimal]
Let $N \geq 4$ and consider the action of $G = (\bZ,+)$ on $X = \bZ/N\bZ$ by translations modulo $N$. The normalized counting measure $\mu$
on $X$ is clearly invariant and ergodic. Let $A = \{0,1\} \subset X$ and note that
\[
\mu(A) = \frac{2}{N} \qand \cR_{A} = N \bZ \cup \big(N\bZ + 1\big) \cup (N\bZ - 1) \subsetneq \bZ.
\]
It is not hard to check that
\[
\underline{d}(\cR_A) = \frac{3}{N} = \frac{3}{2} \mu(A),
\]
but $\cR_A$ is \emph{not} a subgroup of $\bZ$.
\end{example}
 
\begin{example}[Ergodicity of the action is needed]
Given positive real numbers $\delta$ and $\eps$, we shall construct a \emph{non-ergodic} probability measure 
$\mu$ for the shift action by $G = (\bZ,+)$ on the space $2^\bZ$ of all subsets of $\bZ$, endowed with the product topology, such that
\[
\mu(A) < \delta \qand \underline{d}(\cR_A) \leq (1+\eps)\mu(A),
\]
where $A = \{ C \in 2^\bZ \, \mid \, 0 \in C \big\}$, and such that the set of return times $\cR_A$ projects onto every finite quotient of $\bZ$. In 
particular, $\cR_A$ cannot be a subgroup of $\bZ$, nor can it be contained in a subgroup of $\bZ$. 
Here, the exact choice of the sequence $(F_n)$ in $\bZ$ is not so important; for simplicity, we can assume that $F_n = [1,n]$ for all $n \geq 1$. 

The construction of $\mu$ goes along the following lines. Given positive real numbers $\delta$ and $\eps$, we choose $0 < \eta < 1$ such that 
$1 + \eps = \frac{1+\eta}{1-\eta}$, and we pick a strictly increasing sequence $(p_k)$ of prime numbers such that
\begin{equation}
\label{cond1}
\frac{1}{p_1} < \delta \qand \sum_{k \geq 2} \, \frac{1}{p_k} \leq \frac{\eta}{p_1}.
\end{equation}
For every $k \geq 1$, we denote by $\mu_k$ the uniform probability measure on the $\bZ$-orbit of the subgroup $p_k \bZ$ in $2^{\bZ}$ and we note
that $\mu_k(A) = \frac{1}{p_k}$. We now define 
\[
\mu = (1-\eta)\mu_1 + \eta \, \sum_{k \geq 2} \frac{\mu_k}{2^{k-1}},
\]
which is clearly a $\bZ$-invariant \emph{non-ergodic} Borel probability measure on $2^{\bZ}$. One readily checks that
\[
\mu(A) = \frac{1-\eta}{p_1} + \eta \, \sum_{k \geq 2} \frac{1}{p_k 2^{k-1}},
\qand 
\cR_A = \bigcup_{k \geq 1} p_k \bZ \subsetneq \bZ,
\]
whence, 
\[
\frac{1-\eta}{p_1} \leq \mu(A) < \delta
\]
and, thus, by \eqref{cond1} and the choice of $\eta$,
\[
\underline{d}(\cR_A) \leq \sum_{k \geq 1} \frac{1}{p_k} \leq \frac{1+\eta}{p_1} \leq \Big( \frac{1+\eta}{1-\eta} \Big) \, \mu(A) = (1+\eps)\mu(A).
\]
Clearly, $\cR_A$ projects onto every finite quotient of $\bZ$, which finishes our construction.
\end{example}

Our second theorem asserts that the set of return times $\cR_A$ is a still a periodic subset of the group $G$ (that is to say, invariant under a
finite index subgroup) if the weaker upper bound $\underline{d}(\cR_A) < 2 \mu(A)$ holds (this is the ergodic-theoretical analogue of 
\eqref{passtoquotient} described above for \emph{finite} $G$).

\begin{theorem}									
\label{thm2}
Suppose that $\underline{d}(\cR_{A,B}) < \mu(A) + \mu(B)$. Then there exist 
\begin{enumerate}
\item[(i)] a proper finite-index subgroup $G_o < G$ and a homomorphism $\eta$ from $G$ onto the quotient group $G/G_o$, 
\item[(ii)] a non-trivial $G$-factor $\sigma : (X,\mu) \ra (G/G_o,m_{G/G_o})$, where $m_{G/G_o}$ denote the normalized 
counting measure on $G/G_o$ and $G$ acts on $G/G_o$ via $\eta$,
\item[(iii)] a finite subset $M \subset G/G_o$,
\end{enumerate} 
such that $\cR_{A,B} = \eta^{-1}(M)$. Furthermore, there are finite subsets $I_o, J_o \subset G/G_o$ such that
the pair $(I_o,J_o)$ $\sigma$-controls $(A,B)$.
\end{theorem}

An ergodic action $G \acts (X,\mu)$ is \emph{totally ergodic} if every finite-index subgroup of $G$ acts ergodically. We note that if 
$G \acts (X,\mu)$ admits a factor of the form $G/G_o$ for some finite-index subgroup $G_o < G$, then $G_o$ cannot act ergodically
on $(X,\mu)$, whence the $G$-action on $(X,\mu)$ is not totally ergodic. The following corollary of Theorem \ref{thm2} is now immediate. 

\begin{corollary}
\label{cor_toterg}
Suppose that the action $G \acts (X,\mu)$ is totally ergodic. Then, for all measurable subsets $A, B \subset X$ with positive $\mu$-measures, 
\[
\underline{d}(\cR_{A,B}) \geq \min(1,\mu(A) + \mu(B)).
\]
\end{corollary}

\begin{example}["Non-conventional" lower asymptotic density]
If $\bZ \acts (X,\mu)$ is totally ergodic, then the sequence $(F_n)$ of finite subsets of $\bZ$ defined by
\[
F_n = \{ k^2 \, : \, k = 1,\ldots,n \big\}, \quad \textrm{for $n \geq 1$},
\]
satisfies \eqref{pointwise} (the almost sure convergence follows from the work of Bourgain \cite{Bo88}, while the identification of the limit - for totally ergodic actions - follows from the equidistribution (modulo 1) of the sequence $(n^2 \alpha)$, for irrational $\alpha$). In particular, we can conclude from Corollary 
\ref{cor_toterg} that
\[
\varliminf_{n \ra \infty} \frac{|\cR_{A,B} \cap \{1,4,\ldots,n^2\}|}{n} \geq \min(1,\mu(A) + \mu(B)),
\]
for all measurable subsets $A, B \subset X$ with positive $\mu$-measures.
\end{example}

\subsection{The structure of transfer times for totally ergodic actions}

\begin{example}[Sturmian sets]
\label{example_sturmian}
Suppose that $G$ admits a homomorphism into $\bT = \bR/\bZ$ with dense image. We set $X = \bT$ and denote by $\mu$ the normalized
Haar measure on $\bT$. Note that $G$ acts on $X$ via $\tau$. Let $A$ and $B$ be two closed intervals of $\bT$ with $\mu(A) + \mu(B) < 1$
such that the endpoints of the interval $BA^{-1}$ in $\bT$ belong to $\tau(G)$. Then it is not hard to show that
\[
\underline{d}(\cR_{A,B}) = \mu(A) + \mu(B) \qand \cR_{A,B} = \tau^{-1}(BA^{-1}) \subseteq G,
\]
which in particular shows that the lower bound in Corollary \ref{cor_toterg} is attained. Pullbacks to $G$ of closed intervals in $\bT$ under homomorphisms with dense images are often 
called \emph{Sturmian sets} in the literature. 
\end{example}

Our next theorem asserts that under the assumption of total ergodicity, then Example \ref{example_sturmian} is the essentially the only example 
when the lower bound in \ref{cor_toterg} is attained. We stress that this phenomenon does not really have a counterpart for \emph{finite} $G$.

\begin{theorem}
\label{thm3}
Suppose that the action $G \acts (X,\mu)$ is totally ergodic. If 
\[
\underline{d}(\cR_{A,B}) = \mu(A) + \mu(B) < 1, 
\]
then there exist
\begin{enumerate}
\item[(i)] a homomorphism $\eta : G \ra \bT$ with dense image. 
\item[(ii)] a $G$-factor $\sigma : (X,\mu) \ra (\bT,m_{\bT})$, where $m_{\bT}$ denotes the normalized Haar measure on $\bT$
and $G$ acts on $\bT$ via $\eta$.
\item[(iii)] closed intervals $I_o$ and $J_o$ of $\bT$ with $m_{\bT}(I_o) = \mu(A)$ and $m_{\bT}(J_o) = \mu(B)$ 
\end{enumerate}
such that $(I_o,J_o)$ $\sigma$-controls $(A,B)$ and
\[
\cR_{A,B} = \eta^{-1}(J_o I_o^{-1}),
\]
modulo at most two cosets of the subgroup $\ker \eta$.
\end{theorem}

\subsection{On ergodic actions which admit small sets of return times}

We retain the notation and assumptions from the beginning of the introduction. 

\begin{definition}[$C$-doubling actions]
Let $C \geq 1$. We say that $G \acts (X,\mu)$ is \emph{a $C$-doubling action} if for every $\delta > 0$, there 
exists a measurable subset $A \subset X$ with $0 < \mu(A) < \delta$ such that $\underline{d}(\cR_A) \leq C\mu(A)$.
\end{definition}

We note that if the action is $C$-doubling, then it also $C'$-doubling for every $C' \geq C$. \\

In light of our theorems above, it seems natural to ask about the structure of $C$-doubling actions. The following theorem provides a complete 
characterization.

\begin{theorem}
\label{thm4}
Let $C \geq 1$. An ergodic action $G \acts (X,\mu)$ is $C$-doubling if and only if there exist
\begin{enumerate}
\item[(i)] an infinite compact metrizable group $K$ and a homomorphism $\eta : G \ra K$ with dense image. 
\item[(ii)] a $G$-factor $\sigma : (X,\mu) \ra (K,m_K)$, where $m_K$ denotes the normalized Haar measure on $K$
and $G$ acts on $K$ via $\eta$.
\end{enumerate}
Furthermore, 
\begin{itemize}
\item If the identity component $K^o$ of $K$ has infinite index, then the action is $1$-doubling.
\item If the identity component $K^o$ of $K$ has finite index, then the action is $2$-doubling. 
\end{itemize}
\end{theorem}

\begin{remark}
Theorem \ref{thm4} in particular asserts that an ergodic action is $C$-doubling for some $C \geq 1$ if and only if it has an infinite Kronecker 
factor (see e.g. \cite{BF} for definitions).
\end{remark}

The same line of argument as the one leading up to Theorem \ref{thm4} also proves the following result, whose proof we leave to the reader. 
We recall that $G \acts (X,\mu)$ is \emph{weakly mixing} if the diagonal action $G \acts (X \times X,\mu \otimes \mu)$ is ergodic.

\begin{scholium}
There exist measurable subsets $A$ and $B$ of $X$ such that $\underline{d}(\cR_{A,B}) < 1$ if and only if $G \acts (X,\mu)$ is \emph{not} 
weakly mixing. 
\end{scholium}

\subsection{A brief outline of the proofs}

Our first observation is that for any two measurable subsets $A$ and $B$ of $X$ with positive $\mu$-measures, there is a measurable 
$\mu$-conull subset $X_1$ of $X$ such that
\[
\cR_{A,B} = B_x A_x^{-1}, \quad \textrm{for all $x \in X_1$},
\]
where $A_x$ and $B_x$ are the return times of the point $x$ to the sets $A$ and $B$ (see Subsection \ref{subsec:deftransfer} below for notation). 
We then observe in Lemma \ref{lemma_generic} that for some measurable $\mu$-conull subset $X_2 \subset X$, 
\[
\underline{d}(B_xA_x^{-1}) \geq \mu(A_x^{-1}B), \quad \textrm{for all $x \in X_2$},
\]
which puts us in the framework of our earlier paper \cite{BF}. We combine some of the key points of this paper in Lemma \ref{corrprinciple}
below, the outcome of which is that there exist
\begin{itemize}
\item a measurable $G$-invariant $\mu$-conull subset $X_3 \subset X_1 \cap X_2$,
\item a compact and metrizable abelian group $K$ with Haar probability measure $m_K$ and 
a homomorphism $\tau : G \ra K$ with dense image,
\item a $G$-equivariant measurable map $\pi : X_3 \ra K$ such that $\pi_*(\mu |_{X_3}) = m_K$, where $G$ acts on $K$ via $\tau$,
\item two measurable subsets $I$ and $J$ of $K$,
\end{itemize}
such that
\[
\mu(A_x^{-1}B) = m_K(JI^{-1})
\]
and 
\[
A \cap X_3 \subset \pi^{-1}(I) \qand B \cap X_3 \subset \pi^{-1}(J).
\]
If $A = B$, then we can take $I = J$. We see that $\mu(A) \leq m_K(I)$ and $\mu(B) \leq m_K(J)$. 
In the settings of Theorem \ref{thm1}, Theorem \ref{thm2} and Theorem \ref{thm3}, we see that
\[
m_K(I I^{-1}) < \frac{3}{2} m_K(I)
\qand
m_K(JI^{-1}) < m_K(I) + m_K(J)
\]
and
\[
m_K(J I^{-1}) = \min(1,m_K(I) + m_K(J))
\]
respectively. At this point, we use some classical results \cite{Kneser} of Kneser for sumsets in compact abelian groups,
to conclude that the pair $(I,J)$ is "reduced" to a nicer pair $(I_o,J_o)$ in a much "smaller" quotient group $Q$ of $K$ (see 
Definition \ref{def_reduction} for details). The point of all this is that the transfer times $\cR_{A,B}$ is \emph{contained in} the 
transfer times between $I_o$ and $J_o$, which is equal to the set $\eta^{-1}(J_oI_o^{-1})$. Here $\eta : G \ra Q$ is the composition 
of $\tau$ with the quotient map from  $K$ to $Q$. To prove that the sets actually coincide,
we shall use the \emph{overshoot relation}
\begin{equation}
\label{defovershoot}
\mu(A) + \mu(B) \leq m_Q(I_o) + m_Q(J_o) - m_Q(I_o \cap \eta(g)^{-1}J_o), 
\end{equation}
for all $g \in \eta^{-1}(J_o I_o^{-1}) \setminus \cR_{A,B}$. This inequality is proved in Proposition \ref{mainprop}. It turns out that in 
the settings of the theorems above, the sets $I_o$ and $J_o$ have the property that the $m_Q$-measure of the 
intersection $I_o \cap \eta(g)^{-1}J_o$, for $g$ in $\eta^{-1}(J_o I_o^{-1}) \setminus \cR_{A,B}$, is large enough to 
contradict \eqref{defovershoot}, whence we can conclude that $\cR_{A,B} = \eta^{-1}(J_o I_o^{-1})$.

\subsection{Ergodic actions of semi-groups}

Our definition of transfer times between two sets makes sense also for actions by non-invertible maps. Suppose that $S$ is a countable 
abelian semigroup, sitting inside a countable abelian group $G$. If $S$ acts ergodically by measure-preserving maps on a standard 
probability measure space $(X,\mu)$, then, under some technical assumptions (see e.g. \cite{La} for more details in the general setting), one can construct a
so called \emph{natural extension} $(\widetilde{X},\widetilde{\mu})$ of the $S$-action, which is a measure-preserving $G$-action, together with 
a measurable $S$-equivariant map $\rho : \widetilde{X} \ra X$, mapping $\widetilde{\mu}$ to $\mu$. It is not hard to see that if we
set 
\[
\widetilde{A} = \rho^{-1}(A) \qand \widetilde{B} = \rho^{-1}(B),
\]
then 
\[
\cR_{\widetilde{A},\widetilde{B}} \cap S = \{ s \in S \, \mid \, \mu(A \cap s^{-1}B) > 0 \big\},
\]
where the transfer times $\cR_{\widetilde{A},\widetilde{B}}$ are measured with respect to $\tilde{\mu}$. We can now apply our results above
to the $G$-action on the natural extension $(\widetilde{X},\widetilde{\mu})$ (which is ergodic if and only if the semi-group action $S \acts (X,\mu)$ is), and conclude the same results for the 
$S$-action. We leave the details to the interested reader.


\medskip

\subsection{Acknowledgements}
I.S. is grateful to SMRI and the School of Mathematics and Statistics at Sydney University for funding his visit and for their hospitality. 
M.B and A.F wish to thank the organizers of the MFO workshop "Groups, dynamics and approximation", during which parts of this paper
were written, for the invitation.


\section{Preliminaries}

\subsection{Transfer times and action sets}
\label{subsec:deftransfer}
Given a subset $D$ of $X$ and $x \in X$, we define the \emph{set of return time of $x$ to $D$} by
\[
D_x = \{ g \in G \, \mid \, gx \in D \big\} \subset G,
\] 
and we note that $(gD)_x = gD_x$ and $D_x \, g^{-1} = D_{gx}$ for all $g \in G$.
If $F$ is a subset of $G$, then we define the \emph{action set} $FD \subset X$ by
\[
FD = \bigcup_{f \in F} fD,
\]
and we note that $(FD)_x = F D_x$. If $E$ is another subset of $X$, then
\[
(D \cap E)_x = D_x \cap E_x \qand (D \cup E)_x = D_x \cup E_x.
\]
In particular, 
\begin{equation}
\label{eq_DcapE}
D_x \cap g^{-1} E_x = (D \cap g^{-1} E)_x \qand D_x \cup g^{-1} E_x = (D \cup g^{-1} E)_x, \quad \textrm{for all $g \in G$}.
\end{equation}

\subsection{Transfer times as difference sets}

Let $D$ be a measurable subset of $X$, and define
\[
D^{e} = \big\{ x \in X \, | \, D_x = \emptyset \big\} 
\qand
D^{ne} = \big\{ x \in X \, | \, D_x \neq \emptyset \big\}.
\]
We note that $D^{e} = \bigcap_{g \in G} gD^c$ and $D^{ne} = GD$. In particular, $D^e$ and $D^{ne}$ are both measurable and $G$-invariant. 
Since $\mu$ is assumed to be ergodic, we conclude that
\begin{equation}
\label{De}
\mu(D^e) = 1 \quad \textrm{if $\mu(D) = 0$}
\end{equation}
and
\begin{equation}
\label{Dne}
\mu(D^{ne}) = 1 \quad \textrm{if $\mu(D) > 0$}.
\end{equation}

\begin{lemma}
\label{lemma_passtoreturns}
Let $A$ and $B$ be two measurable subsets of $X$ with positive $\mu$-measures, and define
\[
X_1 = \Big( \bigcap_{g \in \cR_{A,B}}  \{ x \in X \, \mid \, A_x \cap g^{-1}B_x \neq \emptyset \big\} \Big)
\cap
\Big( \bigcap_{g \notin \cR_{A,B}}  \{ x \in X \, \mid \, A_x \cap g^{-1} B_x = \emptyset \big\} \Big).
\]
Then $X_1$ is a $G$-invariant measurable $\mu$-conull subset of $X$ and 
\[
\cR_{A,B} = B_x A_x^{-1}, \quad \textrm{for all $x \in X_1$}.
\]
\end{lemma}

\begin{proof}
Measurability and $G$-invariance of $X_1$ is clear, and $\mu$-conullity of $X_1$ readily follows from applying \eqref{De} and \eqref{Dne} to the sets
\[
D(g) := A \cap g^{-1}B, \quad \textrm{for $g \in G$}.
\]
Indeed, $\mu(D(g)) > 0$ if and only if $g \in \cR_{A,B}$, and by \eqref{eq_DcapE}, we have
\[
D(g)^{e} = \big\{ x \in X \, \mid \, A_x \cap g^{-1} B_x = \emptyset \big\}
\qand
D(g)^{ne} = \big\{ x \in X \, \mid \, A_x \cap g^{-1}B_x \neq \emptyset \big\}.
\]
Note that for every $x \in X$,
\begin{eqnarray*}
B_x A^{-1}_x 
&=& 
\{ g \in G \, \mid \, A_x \cap g^{-1} B_x \neq \emptyset \big\} \\[2pt]
&=& 
\{ g \in \cR_{A,B} \, \mid \, A_x \cap g^{-1} B_x \neq \emptyset \big\} \sqcup \{ g \notin \cR_{A,B} \, \mid \, A_x \cap g^{-1} B_x \neq \emptyset \big\}.
\end{eqnarray*}
If $x \in X_1$, then
\[
\{ g \in \cR_{A,B} \, \mid \, A_x \cap g^{-1} B_x \neq \emptyset \big\} = \cR_{A,B}
\qand
\{ g \notin \cR_{A,B} \, \mid \, A_x \cap g^{-1} B_x \neq \emptyset \big\} = \emptyset,
\]
which finishes the proof.
\end{proof}

\subsection{Generic points}

We recall our assumptions on the sequence $(F_n)$ of finite subsets of $G$: For every bounded measurable function $\varphi$ on $X$,
there exists a $\mu$-conull subset $X_\varphi \subset X$ such that 
\[
\lim_n \frac{1}{|F_n|} \sum_{g \in F_n} \varphi(gx) = \int_X \varphi \, d\mu, \quad \textrm{for all $x \in X_\varphi$}.
\]
The points in $X_\varphi$ are said to be \emph{generic} with respect to $\mu$, $\varphi$ and the sequence $(F_n)$.
\begin{lemma}
\label{lemma_generic}
Let $A$ and $B$ be two measurable subsets of $X$ with positive $\mu$-measures. Then there exists a measurable $\mu$-conull subset $X_2 \subseteq X$
such that 
\[
\mu(A_x^{-1}B) \leq \underline{d}(\cR_{A,B}), \quad \textrm{for all $x \in X_2$}.
\]
Furthermore, for every finite subset $L$ of $G$,
\[
\underline{d}(L^{-1}A_x) = \mu(L^{-1}A) \qand \underline{d}(L^{-1}B_x) = \mu(L^{-1}B), 
\]
and for every $g \notin \cR_{A,B}$, 
\[
\underline{d}(A_x \cup g^{-1}B_x) = \mu(A) + \mu(B),
\]
for all $x \in X_2$.
\end{lemma}

\begin{proof}
Given a subset $L \subset G$, we define 
\[
\varphi_L = \chi_{L^{-1}A} \qand \psi_L = \chi_{L^{-1}B} \qand X_L = X_{\varphi_L} \cap X_{\psi_L}.
\]
We note $X_L$ is a measurable $\mu$-conull subset of $X$ and for every $x \in X_L$, 
\begin{equation}
\label{LlimitA}
\underline{d}(L^{-1}A_x) = \lim_n \frac{1}{|F_n|} \sum_{g \in F_n} \chi_{L^{-1}A}(gx) = \mu(L^{-1}A)
\end{equation}
and
\begin{equation}
\label{LlimitB}
\underline{d}(L^{-1}B_x) = \lim_n \frac{1}{|F_n|} \sum_{g \in F_n} \chi_{L^{-1}B}(gx) = \mu(L^{-1}B).
\end{equation}
We now set $X_2' = \bigcap_{L} X_L$, where the intersection is taken over the countable set of all \emph{finite} subsets of $G$. Then
$X_2'$ is a measurable $\mu$-conull subset of $X$, and for every $x \in X_2'$ and for every finite subset $L$ of $A_x$, we have
\[
\underline{d}(A_x^{-1}B_x) \geq \underline{d}(L^{-1}B_x) = \mu(L^{-1}B).
\]
Since $\mu$ is $\sigma$-additive and $L \subset A_x$ is an arbitrary finite set, we can now conclude that 
\[
\underline{d}(A_x^{-1}B_x) \geq \mu(A_x^{-1}B) \quad \textrm{for all $x \in X_2'$}. 
\]
By Lemma \ref{lemma_passtoreturns}, there exists a measurable $\mu$-conull subset $X_1 \subseteq X$ such that 
$\cR_{A,B} = B_x A_x^{-1}$ for all $x \in X_1$, and thus, since $G$ is abelian,
\[
\underline{d}(\cR_{A,B}) = \underline{d}(A_x^{-1}B_x) \geq \mu(A_x^{-1}B), \quad \textrm{for all $x \in X_1 \cap X_2'$}.
\]
Let $X_2 = X_1 \cap X_2'$ and pick $x \in X_2$. We note that if $g \notin \cR_{A,B} = B_x A_x^{-1}$, then $A_x \cap g^{-1}B_x = \emptyset$,
whence 
\begin{eqnarray*}
\underline{d}(A_x \cup g^{-1}B_x) 
&=& 
\varliminf_n \Big( \frac{|A_x \cap F_n|}{|F_n|} +  \frac{|(g^{-1}B)_x \cap F_n|}{|F_n|} \Big) \\
&=&
\mu(A) + \mu(B) = \mu(A \cup g^{-1}B),
\end{eqnarray*}
by \eqref{LlimitA} and \eqref{LlimitB} (applied to the sets $L = \{e\}$ and $L = \{g\}$ respectively), since the limits of each term exist 
(the last identity follows from the fact that $\mu(A \cap g^{-1}B) = 0$ if $g \notin \cR_{A,B}$). Since $x \in X_2$ is arbitrary, this finishes the proof. 
\end{proof}

\begin{corollary}
\label{cor_max}
For all measurable subsets $A$ and $B$ of $X$, we have
\[
\underline{d}(\cR_{A,B}) \geq \max(\mu(A),\mu(B)).
\]
\end{corollary}

\begin{proof}
By Lemma \ref{lemma_generic}, there is a measurable $\mu$-conull subset $X_2$ of $X$ such that
\[
\underline{d}(\cR_{A,B}) \geq \mu(A_x^{-1}B) \geq \mu(B), \quad \textrm{for all $x \in X_2$}.
\]
Since the roles of $A$ and $B$ are completely symmetric, this proves the corollary.
\end{proof}

\subsection{A correspondence principle for action sets}

The key ingredient in the proofs of Theorem \ref{thm1}, Theorem \ref{thm2} and Theorem \ref{thm3} is the following merger of a series of 
observations made by the first two authors in \cite{BF}. We outline the anatomy of this merger in the proof below. The rough idea is the 
action sets in an arbitrary ergodic $G$-action can be controlled by sets in an isometric factor (that is to say, a compact group, on which $G$
acts by translations via a homomorphism from $G$ into the compact group with dense image).

\begin{lemma}
\label{corrprinciple}
Let $A$ and $B$ be measurable subsets of $X$ with positive $\mu$-measures. Then there exist 
\vspace{0.1cm}
\begin{itemize}
\item a $G$-invariant measurable $\mu$-conull subset $X_3 \subseteq X$,
\item a compact and metrizable abelian group $K$ with Haar probability measure $m_K$, a homomorphism $\tau : G \ra K$ with dense image,
and two measurable subsets $I$ and $J$ of $K$,
\item a $G$-equivariant measurable map $\pi : X_3 \ra K$ such that $\pi_*(\mu |_{X_3}) = m_K$, where $G$ acts on $K$ via $\tau$,
\end{itemize} 
\vspace{0.1cm}
such that 
\[
A \cap X_3 \subseteq \pi^{-1}(I) \qand B \cap X_3 \subseteq \pi^{-1}(J)
\]
and
\[
\mu(A_x^{-1}B) = m_K(I^{-1}J)
\qand
A_x^{-1}(B \cap X_3) \subseteq \pi^{-1}(\pi(x) I^{-1}J), 
\]
for all $x \in X_3$. In the case when $A = B$, we can take $I = J$. Finally, if $G \acts (X,\mu)$ is totally ergodic, then $K$ must be connected.
\end{lemma}

\begin{remark}
If $I$ and $J$ are Borel measurable subsets of $K$, then their difference set $I^{-1}J$ might fail to be Borel measurable. However, since 
$I^{-1}J$ is the image of the Borel measurable subset $I \times J$ in $K \times K$ under the continuous map $(k_1,k_2) \mapsto k_1^{-1}k_2$,
we see that $I^{-1}J$ is an analytic set, so in particular measurable with respect to the \emph{completion} of the Borel $\sigma$-algebra of $K$ with
respect to $m_K$, and thus the expression $m_K(I^{-1}J)$ is well-defined.
\end{remark}

\begin{proof}
By \cite[Lemma 5.3]{BF}, there exists a $G$-invariant measurable $\mu$-conull subset $X_3' \subset X$ such that 
\[
\mu(A_x^{-1}B) = \mu \otimes \mu(G(A \times B)), \quad \textrm{for all $x \in X_3'$}.
\]
By \cite[Theorem 5.1]{BF}, there exist
\begin{itemize}
\item a measurable $G$-invariant $\mu$-conull subset $X_3'' \subset X$,
\item a compact and metrizable abelian group $K$ with Haar probability measure $m_K$ and 
a homomorphism $\tau : G \ra K$ with dense image,
\item a $G$-equivariant measurable map $\pi : X_3'' \ra K$ such that $\pi_*(\mu |_{X''_3}) = m_K$, where $G$ acts on $K$ via $\tau$,
\item two measurable subsets $I$ and $J$ of $K$,
\end{itemize}
such that
\[
\mu(G(A \times B)) = m_K(I^{-1}J)
\]
and 
\[
A \cap X_3'' \subset \pi^{-1}(I) \qand B \cap X_3'' \subset \pi^{-1}(J).
\]
It follows from the proof of \cite[Theorem 5.1]{BF} that if $A = B$, then we can take $I = J$. Since the set $X_3''$ is $G$-invariant, 
we see that 
\[
A_x \subset \pi^{-1}(I)_x = \tau^{-1}(I\pi(x)^{-1}), \quad \textrm{for all $x \in X_3''$},
\]
whence
\[
A_x^{-1}(B \cap X_3'') \subset A_x^{-1}\pi^{-1}(J) = \pi^{-1}(\tau(A_x)^{-1}J) \subset \pi^{-1}(\pi(x) I^{-1}J).
\]
Let $X_3 := X_3' \cap X_3''$ and note that $X_3$ is $G$-invariant and $\mu$-conull. Thus the proof is finished modulo our assertion about total ergodicity. Suppose that $K$ is not connected. Then there is an open subgroup $U$ of $K$, and $G_o = \tau^{-1}(U)$ is a finite-index 
subgroup of $G$. We note that $C := \pi^{-1}(U)$ is a $G_o$-invariant measurable subset of $X$, with positive $\mu$-measure, but
which cannot be $\mu$-conull, since it does not map onto $K$ under $\pi$ (modulo $\mu$-null sets). We conclude that $G \acts (X,\mu)$
is not totally ergodic.
\end{proof}

\subsection{Putting it all together}

Let $K$ and $Q$ be compact groups with Haar probability measures $m_K$ and $m_Q$ respectively and 
suppose that there is a continuous homomorphism $p$ from $K$ onto $Q$. 

\begin{definition}[Pair reduction]
\label{def_reduction}
Let $(I,J)$ and $(I_o,J_o)$ be two pairs of measurable subsets of $K$ and $Q$ respectively. We say 
that \emph{$(I,J)$ reduces to $(I_o,J_o)$ with respect to $p$} if
\[
I \subset p^{-1}(I_o) \qand J \subset p^{-1}(J_o) \qand m_K(J I^{-1}) = m_Q(J_o I_o^{-1}).
\]
\end{definition}

This notion is quite useful when we now summarize our discussion above. 

\begin{proposition}[A correspondence principle for transfer times]
\label{mainprop}
Let $A$ and $B$ be measurable subsets of $X$ with positive $\mu$-measures. Then there exist 
\vspace{0.1cm}
\begin{itemize}
\item a compact and metrizable abelian group $K$ with Haar probability measure $m_K$,
\item a homomorphism $\tau : G \ra K$ with dense image,
\item a pair $(I,J)$ of measurable subsets of $K$,
\end{itemize} 
\vspace{0.1cm}
which satisfy
\[
\mu(A) \leq m_K(I) \qand \mu(B) \leq m_K(B) \qand m_K(JI^{-1}) \leq \underline{d}(\cR_{A,B}).
\]
Furthermore, suppose that $Q$ is a compact group and $p : K \ra Q$ is a continuous surjective homomorphism. 
If $(I_o,J_o)$ is a pair of measurable subsets of $Q$ such that $(I,J)$ reduces to $(I_o,J_o)$ with respect to $p$,
then
\[
\cR_{A,B} \subseteq \tau_p^{-1}(J_o I_o^{-1}),
\]
where $\tau_p = p \circ \tau$, and for all $g \in \tau_p^{-1}(J_oI_o^{-1}) \setminus \cR_{A,B}$, we have
\[
\mu(A) + \mu(B) \leq m_Q(I_o) + m_Q(J_o) - m_Q(I_o \cap \tau_p(g)^{-1}J_o),
\]
Moreover, there exists a $G$-factor map $\sigma : (X,\mu) \ra (Q,m_Q)$, where $G$ acts on $Q$ via $\tau_p$, such that
\[
A \subseteq \sigma^{-1}(I_o) \qand B \subseteq \sigma^{-1}(J_o), \quad \textrm{modulo $\mu$-null sets}.
\]
In the case when $A = B$, we can take $I = J$.
\end{proposition}

\begin{proof}
By Lemma \ref{corrprinciple}, we can find a $G$-invariant measurable $\mu$-conull subset $X_3 \subseteq X$, a compact and 
metrizable abelian group $K$ with Haar probability measure $m_K$, a homomorphism $\tau : G \ra K$ with dense image,
and two measurable subsets $I$ and $J$ of $K$, a $G$-equivariant measurable map $\pi : X_3 \ra K$ such that 
$\pi_*(\mu |_{X_3}) = m_K$, where $G$ acts on $K$ via $\tau$, such that 
\begin{equation}
\label{incl1}
A \cap X_3 \subseteq \pi^{-1}(I) \qand B \cap X_3 \subseteq \pi^{-1}(J)
\end{equation}
and
\[
m_K(JI^{-1}) \leq \mu(A_x^{-1}B) 
\qand
A_x^{-1}(B \cap X_3) \subseteq \pi^{-1}(\pi(x) I^{-1}J), 
\]
for all $x \in X_3$. Furthermore, by Lemma \ref{lemma_passtoreturns} and Lemma \ref{lemma_generic}, there exist measurable 
$\mu$-conull subsets $X_1$ and $X_2$ of $X$
such that 
\[
\cR_{A,B} = B_x A_x^{-1} \qand \mu(A_x^{-1}B) \leq \underline{d}(\cR_{A,B})
\]
and, for every $g \notin \cR_{A,B}$,
\begin{equation}
\label{recall_overshoot}
\underline{d}(A_x \cup g^{-1}B_x) = \mu(A) + \mu(B) = \mu(A \cup g^{-1}B),
\end{equation}
for all $x \in X_1 \cap X_2$. In particular, since $X_1 \cap X_2 \cap X_3$ is a $\mu$-conull subset of $X$, and thus non-empty, we have
\[
\mu(A) \leq m_K(I) \qand \mu(B) \leq m_K(J) \qand m_K(JI^{-1}) \leq \underline{d}(\cR_{A,B}).
\]
Let us now assume that $Q$ is a compact group, $p : K \ra Q$ is a continuous surjective homomorphism and $I_o$ and $J_o$ are 
measurable subsets of $Q$ such that $(I,J)$ reduces to $(I_o,J_o)$. We recall that this means that 
\[
I \subset p^{-1}(I_o) \qand J \subset p^{-1}(J_o) \qand m_K(JI^{-1}) = m_Q(J_o I_o^{-1}).
\]
Hence, $J I^{-1} \subset p^{-1}(J_o I_o^{-1})$, and 
\[
m_Q(J_oI_o^{-1}) \leq \mu(A_x^{-1}B) 
\qand
A_x^{-1}(B \cap X_3) \subseteq \pi^{-1}(\pi(x) p^{-1}(J_o I_o^{-1})).
\]
We note that we can write
\[
\pi^{-1}(\pi(x) p^{-1}(J_o I_o^{-1})) = \sigma^{-1}(\sigma(x) J_o I_o^{-1}), 
\]
for all $x \in X_3$, where $\sigma = p \circ \pi$, and thus
\begin{equation}
A_x^{-1}(B \cap X_3) \subseteq \sigma^{-1}(\sigma(x) J_o I_o^{-1}), \quad \textrm{for all $x \in X_3$}.
\end{equation} 
The map $\sigma$ is a  $G$-factor map from $(X,\mu)$ to $(Q,m_Q)$, where $G$ acts on 
$Q$ via $\tau_p = p \circ \tau$, and it follows from \eqref{incl1} that
\begin{equation}
\label{ABincl}
A \cap X_3 \subset \sigma^{-1}(I_o) \qand B \cap X_3 \subset \sigma^{-1}(J_o).
\end{equation}
It remains to prove that
\[
\cR_{A,B} \subseteq \tau_p^{-1}(J_o I_o^{-1}),
\]
and that for every $g \in \tau_p^{-1}(J_o I_o^{-1}) \setminus \cR_{A,B}$, we have
\begin{equation}
\label{overshoot}
\mu(A) + \mu(B) \leq m_Q(I_o) + m_Q(J_o) - m_Q(I_o \cap \tau_p(g)^{-1}J_o).
\end{equation}
To prove the inclusion, we first note that since $X_3$ is $G$-invariant, we have
\begin{eqnarray*}
(A_x^{-1}(B \cap X_3))_x
&=& 
A_x^{-1} B_x \subset \sigma^{-1}(\sigma(x)J_o I_o^{-1})_x \\
&=& \tau_p^{-1}(\sigma(x)J_o I_o^{-1}\sigma(x)^{-1}) = \tau_p^{-1}(J_o I_o^{-1}),
\end{eqnarray*}
for all $x \in X_3$. To prove \eqref{overshoot}, we recall from \eqref{recall_overshoot} that if $g \notin \cR_{A,B}$, then
\[
\underline{d}(A_x \cup g^{-1}B_x)
= \mu(A) + \mu(B) = \mu(A \cup g^{-1}B),
\]
whence, by \eqref{ABincl},
\begin{eqnarray*}
\underline{d}(A_x \cup g^{-1}B_x)
&=& \mu(A) + \mu(B) = \mu(A \cup g^{-1}B) \\
&\leq & 
\mu(\sigma^{-1}(I_o) \cup g^{-1} \sigma^{-1}(J_o)) 
= 
m_Q(I_o \cup \tau_p(g)^{-1}J_o) \\
&=& 
m_Q(I_o) + m_Q(J_o) - m_Q(I_o \cap \tau_p(g)^{-1}J_o),
\end{eqnarray*}
which finishes the proof.
\end{proof}

\subsection{Classical product set theorems in compact groups}

We shall use the following two results about product sets in compact groups due to Kneser in his very influential paper \cite{Kneser}.

\begin{theorem}{\cite[Satz 1]{Kneser}}
\label{thm_Kneser1}
Let $K$ be a compact and metrizable abelian group with Haar probability measure $m_K$ and suppose that $I$ and $J$ are measurable subsets of $K$
with positive $m_K$-measures such that
\[
m_K(JI^{-1}) < m_K(I) + m_K(J).
\]
Then $JI^{-1}$ is a clopen subset of $K$, and there exist 
\vspace{0.1cm}
\begin{itemize}
\item a finite group $Q$ and a homomorphism $p$ from $K$ onto $Q$.
\item a pair $(I_o,J_o)$ of subsets of $Q$ with 
\[
m_Q(J_o I_o^{-1}) = m_Q(J_o) + m_Q(I_o) - m_Q(\{e_Q\}),
\]
\end{itemize}  
\vspace{0.1cm}
such that $(I,J)$ reduces to $(I_o,J_o)$ with respect to $p$. If $I = J$, we can take $I_o = J_o$.
\end{theorem}

\begin{corollary}
\label{cor_kneser1}
Let $K$ be a compact and metrizable abelian group with Haar probability measure $m_K$ and assume that $I$ is a measurable
subset of $K$ with positive $m_K$-measure such that 
\[
m_K(II^{-1}) < \frac{3}{2} m_K(I).
\]
Then there exist a finite group $Q$, a surjective homomorphism $p : K \ra Q$ and a point $q \in Q$ such that $(I,I)$ reduces
to $(\{q\},\{q\})$ with respect to $p$. In particular, $II^{-1}$ is an open subgroup of $K$.
\end{corollary}

\begin{proof}
By Theorem \ref{thm_Kneser1}, there exist a finite group $Q$, a homomorphism $p$ from $K$ onto $Q$ and
a subset $I_o$ of $Q$, such that 
\[
I \subset p^{-1}(I_o) \qand II^{-1} = p^{-1}(I_o I_o^{-1}) \qand m_Q(I_o I_o^{-1}) = 2m_Q(\tilde{I}) - m_Q(\{e_Q\}).
\]
Since $m_K(II^{-1}) < \frac{3}{2} m_K(I)$, we conclude that
\begin{equation}
\label{Qbnd}
m_Q(I_o I_o^{-1}) = 2m_Q(I_o) - m_Q(\{e_Q\}) < \frac{3}{2} m_Q(I_o),
\end{equation}
whence $m_Q(I_o I_o^{-1}) < \frac{3}{2} m_Q(\{e_Q\})$. Since $I_o$ is non-empty, we conclude that $I_o I_o^{-1}$ must be a
point. Hence $I_o = \{q\}$ for some $q \in Q$.
\end{proof}

If $K$ is connected and non-trivial, then there are no proper clopen subsets of $K$, whence the assumed upper bound in Theorem \ref{thm_Kneser1}
can never occur. 

\begin{corollary}
\label{cor_kneser2}
Let $K$ be a compact, metrizable and connected abelian group with Haar probability measure $m_K$. Then, for all measurable subsets 
$I$ and $J$ of $K$, 
\[
m_K(JI^{-1}) \geq \min\big(1,m_K(I) + m_K(J)\big).
\]
\end{corollary}

In the connected case, Kneser further characterized the pairs of measurable subsets of the group for which the lower bound in Corollary \eqref{cor_kneser2}
is attained. We denote by $\bT$ the group $\bR/\bZ$ endowed with the quotient topology.

\begin{theorem}{\cite[Satz 2]{Kneser}}
\label{thm_Kneser2}
Let $K$ be a compact, metrizable and connected abelian group with Haar probability measure $m_K$. If $I$ and $J$ are measurable
subsets of $K$ such that
\[
m_K(JI^{-1}) = m_K(I) + m_K(J) \leq 1,
\]
then there exist
\vspace{0.1cm}
\begin{itemize}
\item a continuous homomorphism $p$ from $K$ onto $\bT$,
\item closed intervals $I_o$ and $J_o$ in $\bT$ with 
\[
m_\bT(I_o) = m_K(I) \qand m_{\bT}(J_o) = m_{K}(J),
\]
\end{itemize} 
\vspace{0.1cm}
such that $(I,J)$ reduces to $(I_o,J_o)$ with respect to $p$.
\end{theorem}

\section{Proof of Theorem \ref{thm1} and Theorem \ref{thm2}}

Let $A$ and $B$ be measurable subsets of $X$ with positive $\mu$-measures. The first assertion of Theorem \ref{thm1} is contained
in Corollary \ref{cor_max}. Let us assume that either 
\begin{equation}
\label{case1}
A = B \qand \underline{d}(\cR_{A}) < \frac{3}{2} \mu(A)
\end{equation}
or
\begin{equation}
\label{case2}
\underline{d}(\cR_{A,B}) < \mu(A) + \mu(B).
\end{equation}
By the first part of Proposition \ref{mainprop}, there exist
\begin{itemize}
\item a compact and metrizable abelian group $K$ with Haar probability measure $m_K$,
\item a homomorphism $\tau : G \ra K$ with dense image,
\item a pair $(I,J)$ of measurable subsets of $K$,
\end{itemize} 
which satisfy
\[
\mu(A) \leq m_K(I) \qand \mu(B) \leq m_K(B) \qand m_K(JI^{-1}) \leq \underline{d}(\cR_{A,B}).
\]
In the case \eqref{case1}, which corresponds to Theorem \ref{thm1}, we can take $I = J$, and thus
\[
m_K(II^{-1}) \leq \underline{d}(\cR_{A}) < \frac{3}{2} \, m_K(I).
\]
and in the case \eqref{case2}, which corresponds to Theorem \ref{thm2}, we have
\[
m_K(JI^{-1}) \leq \underline{d}(\cR_{A,B}) < \mu(A) + \mu(B) \leq m_K(I) + m_K(J).
\]
In both cases, Theorem \ref{thm_Kneser1} tells us that there exist a finite group $Q$, a continuous surjective homomorphism 
$p : K \ra Q$ and a pair $(I_o,J_o)$ of subsets of $Q$ such that $(I,J)$ reduces to $(I_o,J_o)$ with respect to $p$.
By Proposition \ref{mainprop}, this implies that
\[
\cR_{A,B} \subset \tau_p^{-1}(J_o I_o^{-1}),
\]
and that for all $g \in \tau_p^{-1}(J_o I_o^{-1}) \setminus \cR_{A,B}$, 
\[
\mu(A) + \mu(B) \leq m_Q(I_o) + m_Q(J_o) - m_Q(I_o \cap \tau_p(g)^{-1}J_o).
\]
In the case \eqref{case1}, Corollary \ref{cor_kneser1} further asserts that $I_o = J_o = \{q\}$ for some point $q \in Q$, whence 
$I_o I_o^{-1} = e_Q$ and thus we can conclude from above that $\cR_{A} \subset G_o := \ker \tau_p$, and
\begin{equation}
\label{eQsize}
m_Q(\{e_Q\}) = m_K(I_o I_o^{-1}) \leq \underline{d}(\cR_{A,B}) < \frac{3}{2} \mu(A).
\end{equation}
Since $Q$ is finite, $G_o$ has finite index in $G$ and for every $g \in G_o \setminus \cR_A$, we have
\[
m_Q(I_o \cap \tau_p(g)^{-1}I_o) \geq m_Q(\{e_Q\}).
\]
Hence,
\[
2\mu(A) \leq 2m_Q(I_o) - m_Q(I_o \cap \tau_p(g)^{-1}I_o) \leq m_Q(\{e_Q\}).
\]
The last inequality clearly contradicts \eqref{eQsize}, so we conclude that $G_o = \cR_A$, which finishes the proof of Theorem \ref{thm1}. \\

In the case of \eqref{case2}, Theorem \ref{thm_Kneser1} asserts that the pair $(I_o,J_o)$ in $Q$ satisfies
\[
m_Q(J_o I_o^{-1}) = m_Q(I_o) + m_Q(J_o) - m_Q(\{e_Q\}),
\]
whence
\begin{equation}
\label{sizeQ2}
m_Q(I_o) + m_Q(J_o) - m_Q(\{e_Q\}) < \underline{d}(\cR_{A,B}) < \mu(A) + \mu(B).
\end{equation}
By Proposition \ref{mainprop}, this implies that
\[
\cR_{A,B} \subset \tau_p^{-1}(J_o I_o^{-1})
\]
and for all $g \in \tau_p^{-1}(J_o I_o^{-1}) \setminus \cR_{A,B}$, we have
\[
\mu(A) + \mu(B) \leq m_Q(I_o) + m_Q(J_o) - m_Q(I_o \cap \tau_p(g)^{-1}J_o).
\]
Since $g \in \tau_p^{-1}(J_o I_o^{-1})$ and $Q$ is finite, we have
\[
m_Q(I_o \cap \tau_p(g)^{-1}J_o) \geq m_Q(\{e_Q\}),
\]
whence 
\[
\mu(A) + \mu(B) \leq m_Q(I_o) + m_Q(J_o) - m_Q(\{e_Q\}),
\]
which clearly contradicts \eqref{sizeQ2}. We conclude that $\tau_p^{-1}(J_o I_o^{-1}) \setminus \cR_{A,B}$ is empty, and thus
\[
\cR_{A,B} = \tau_p^{-1}(J_o I_o^{-1}) = MG_o,
\] 
where $G_o = \ker \tau_p$, and $M$ is a finite subset of $G$ whose image under $\tau_p$ equals $J_o I_o^{-1}$. 
Since $Q$ is finite, $G_o$ has finite index in $G$. This proves Theorem \ref{thm2} (with $\eta = \tau_p$).

\section{Proof of Theorem \ref{thm3}}

Suppose that $G \acts (X,\mu)$ is totally ergodic. Let $A$ and $B$ be measurable subsets of $X$ with positive $\mu$-measures, 
and assume that 
\[
\underline{d}(\cR_{A,B}) = \mu(A) + \mu(B) < 1.
\]
By the first part of Proposition \ref{mainprop}, we can find
\begin{itemize}
\item a compact and metrizable abelian group $K$ with Haar probability measure $m_K$,
\item a homomorphism $\tau : G \ra K$ with dense image,
\item a pair $(I,J)$ of measurable subsets of $K$,
\end{itemize} 
which satisfy
\[
\mu(A) \leq m_K(I) \qand \mu(B) \leq m_K(B) \qand m_K(JI^{-1}) \leq \underline{d}(\cR_{A,B}).
\]
Furthermore, since $G \acts (X,\mu)$ is totally ergodic, $K$ must be connected. In particular, by
Corollary \ref{cor_kneser2}, 
\[
\min(1,m_K(I) + m_K(J)) \leq m_K(JI^{-1}) \leq \underline{d}(\cR_{A,B}) \leq \mu(A) + \mu(B) \leq m_K(I) + m_K(J).
\]
If $m_K(I) + m_K(J) \geq 1$, then $m_K(JI^{-1}) = 1$, whence $\mu(A) + \mu(B) \geq 1$, which we have assumed away.
Hence, $m_K(I) + m_K(J) < 1$, and thus $\mu(A) = m_K(I)$ and $\mu(B) = m_K(J)$, and
\[
m_K(JI^{-1}) = m_K(I) + m_K(J) < 1.
\] 
Theorem \ref{thm_Kneser2} now asserts 
that there is a continuous surjective homomorphism $p : K \ra \bT$ and closed intervals $I_o$ and $J_o$ of $\bT$
such that 
\[
\mu(A) = m_K(I) = m_\bT(I_o) \qand \mu(B) = m_K(J) = m_\bT(J_o),
\]
and $(I,J)$ reduces to the pair $(I_o,J_o)$ with respect to $p$. Hence, by the second part of Proposition \ref{mainprop},
\[
\cR_{A,B} \subseteq \tau_p^{-1}(J_o I_o^{-1})
\]
and for all $g \in \tau_p^{-1}(J_o I_o^{-1}) \setminus \cR_{A,B}$, we have
\[
\mu(A) + \mu(B) \leq m_\bT(I_o) + m_\bT(J_o) - m_\bT(I_o \cap \tau_p(g)^{-1}J_o).
\]
We conclude that
\[
m_\bT(I_o \cap \tau_p(g)^{-1}J_o) = 0, \quad \textrm{for all $g \in \tau_p^{-1}(J_o I_o^{-1}) \setminus \cR_{A,B}$}.
\]
Note that $J_o I_o^{-1}$ is a closed interval in $\bT$. Hence $m_\bT(I_o \cap \tau_p(g)^{-1}J_o) = 0$ for some 
$g \in \tau_p^{-1}(J_o I_o^{-1})$ if and only if $\tau_p(g)$ is one of the endpoints of this interval. In other words,
$\cR_{A,B}$ can only differ from the Sturmian set $\tau_p^{-1}(J_o I_o^{-1})$ by at most two cosets of the 
subgroup $\ker \tau_p$. 

Furthermore, since $\mu(A) = m_Q(I_o)$ and $\mu(B) = m_Q(J_o)$, the last part of Proposition \ref{mainprop} asserts
that there is a $G$-factor map $\sigma : (X,\mu) \ra (\bT,m_{\bT})$, where $G$ acts on $\bT$ via $\tau_p$, such that 
\[
A = \sigma^{-1}(I_o) \qand B = \sigma^{-1}(J_o),
\]
modulo $\mu$-null sets. This finishes the proof of Theorem \ref{thm3} (with $\eta = \tau_p$).

\section{Proof of Theorem \ref{thm4}}

Let us first assume that $G \acts (X,\mu)$ is $C$-doubling for some $C \geq 1$. Then, for every $n \geq 1$, there is 
a measurable subset $A_n \subset X$ such that 
\[
0 < \mu(A_n) < \frac{1}{n} \qand \underline{d}(\cR_{A_n}) \leq C\mu(A_n) < \frac{C}{n}.
\]
To avoid trivialities, we shall from now on assume that $n > C$. By Lemma \ref{corrprinciple}, we can find a (non-trivial) compact metrizable 
group $K_n$, a homomorphism $\eta_n : G \ra K_n$ with dense image, a $G$-factor map $\pi_n : (X,\mu) \ra (K_n,m_{K_n})$ and 
a measurable subset $I_n \subset K_n$ such that 
\[
A_n \subset \pi_n^{-1}(I_n) \quad \textrm{modulo null sets} \qand m_{K_n}(I_n^{-1} I_n) \leq \frac{C}{n}, \quad \textrm{for all $n \geq 1$}.
\]
In particular, $m_{K_n}(I_n) \leq \mu(A_n) < \frac{1}{n}$. Let $K$ denote the closure in $\prod_n K_n$ of the diagonally embedded 
subgroup $\{ (\eta_n(g)) \, \mid \, g \in G \big\}$. We note that $\pi = (\pi_n) : (X,\mu) \ra (K,m_K)$ is a $G$-factor map, where $G$
acts on $K$ via $\eta = (\eta_n)$. Since the pull-backs to $K$ of the sets $I_n$ provide measurable subsets of $K$ with arbitrarily small $m_K$-measures, we see that $K$ must be infinite. \\

Let us now assume that there exist
\begin{enumerate}
\item[(i)] an infinite compact metrizable group $K$ and a homomorphism $\eta : G \ra K$ with dense image. 
\item[(ii)] a $G$-factor $\sigma : (X,\mu) \ra (K,m_K)$, where $m_K$ denotes the normalized Haar measure on $K$
and $G$ acts on $K$ via $\eta$.
\end{enumerate}
We wish to prove that $(X,\mu)$ is $C$-doubling for some $C \geq 1$. Since $G \acts (K,m_K)$ is a $G$-factor of $(X,\mu)$, it is clearly 
enough to prove that $G \acts (K,m_K)$ is $C$-doubling. If $K^o$ has infinite index in $K$, then $K/K^o$ is an infinite totally disconnected group,
and thus we can find a decreasing sequence $(U_n)$ of open subgroups of $K$ with $m_K(U_n) = \frac{1}{[K : U_n]} < \frac{1}{n}$ for all $n$. 
Since $\cR_{U_n} = \eta^{-1}(U_n)$, we have
\[
\underline{d}(\cR_{U_n}) = \frac{1}{[K : U_n]} = m_K(U_n), \quad \textrm{for all $n \geq 1$},
\]
which shows that $G \acts (X,\mu)$ is $1$-doubling (we are using here that the sequence $(F_n)$ also satisfy \eqref{pointwise} for all bounded 
measurable functions on $K$).  If $K^o$ has finite index in $K$, then $K^o$ is an open subgroup, and thus has 
positive $m_K$-measure. Fix a  non-trivial continuous character $\chi : K^o \ra \bT$, and note that by connectedness, $\chi$ is onto. Set
\[
I_n = \chi^{-1}\Big(\Big[-\frac{1}{2n},\frac{1}{2n}\Big]\Big) \subset K^o \subset K, \quad \textrm{for $n \geq 1$}.
\]
Then, $m_K(I_n) = \frac{m_K(K^o)}{n}$ for all $n$, and it is not hard to show that 
\[
\underline{d}(\cR_{I_n}) = \underline{d}(\eta^{-1}(I_n - I_n)) \leq 2 m_K(I_n), \quad \textrm{for all $n$},
\]
whence $G \acts (X,\mu)$ is $2$-doubling.

\end{document}